\documentclass{amsproc}

\newtheorem{theorem}{Theorem}[section]
\newtheorem{lemma}[theorem]{Lemma}
\usepackage{mathrsfs}
\usepackage{euscript}
\usepackage{setspace}
\usepackage{amssymb}
\usepackage{etoolbox}
\usepackage[matha,mathb,mathx]{mathabx}
\usepackage{tikz-cd}
\tikzset{%
    symbol/.style={%
        draw=none,
        every to/.append style={%
            edge node={node [sloped, allow upside down, auto=false]{$#1$}}}
    }
}
\theoremstyle{definition}
\newtheorem{definition}[theorem]{Definition}
\newtheorem{example}[theorem]{Example}

\theoremstyle{remark}
\newtheorem{remark}[theorem]{Remark}

\numberwithin{equation}{section}

\begin{document}
\title[DCGCA and Vector Bundles with Inner Structures]{Differential Calculus over Graded Commutative Algebras
and Vector Bundles with Inner Structures}

\author{Jacob Kryczka}
\address{Yanqi Lake Beijing Institute of Mathematical Sciences and Applications, Beijing, China}

\email{jkryczka@bimsa.cn}

\date{August 31, 2022}

\dedicatory{This paper is dedicated to the memory of my teacher, Alexander M. Vinogradov.}


\begin{abstract}
In this note we highlight a common origin for many ubiquitous geometric structures, as well as several new ones by using only the functors of differential calculus in A.M Vinogradov's original sense, adapted to special classes of (graded) commutative algebras. Special attention is given to the particularly simple cases of diole and triole algebras and we show the latter environment is the appropriate one to describe calculus in vector bundles in the presence of a vector-valued fiber metric.
\end{abstract}

\maketitle
\tableofcontents
\specialsection*{Introduction}
It is common knowledge nowadays that one may study the geometric datum of a vector bundle $E\rightarrow M$ over a smooth manifold in terms of algebraic data via its $A:=C^{\infty}(M)$-module of smooth sections. What is not so commonly encountered in the mathematical literature is the idea that the standard differential geometry of vector bundles and the calculus of their linear connections appears as a particular aspect of differential calculus over commutative algebras in the category of geometric\footnote{We do not recall their definition here, instead referring to \cite[{\normalfont\S 0.3.2}]{VinFat} or \cite[{\normalfont\S 11.11}]{Nestruev} for details.} modules over smooth function
algebras \cite{Nestruev,VinFat} (see also the Second Edition \cite{Nestruev2}).

The expressive capacity and universality of this approach, first appearing in \cite{VinLogic}, is apparent in that it establishes, among other things, a common language for the theory of integration interpreted in terms of de
Rham cohomology, as well as traces of endomorphisms that permit a unified description of the cohomology theories associated to flat connections.

Yet another attractive feature of this functorial and algebraic approach is that it offers a direct algorithm for developing both algebraic
Hamiltonian mechanics and Lagrangian mechanics in a variety of contexts, for instance, over graded commutative algebras, by passing to notions of graded
derivations, graded differential operators etc. \cite{KraVin,Ver}.

This generalization to the graded context has proven to be very powerful, for instance it has been used to deduce a conceptual meaning of Riemannian geometry and the calculus of tensor fields in the context of differential calculus over graded commutative algebras via the notion of iterated differential forms \cite{IteratedForms}.
\vspace{.5mm}

What is even less known, it seems, is that these facts arise naturally from a mathematical formalization of the notion of \emph{classical observability} on smooth manifolds (configuration spaces) appearing in physics. Unfortunately we have no space to elaborate such ideas, and instead encourage the interested reader to consult the expository article \cite{Zoo} and references therein for explanations. See also the Preface to \cite{Nestruev2} for a concise account.
\vspace{2mm}

\subsubsection{The `Fat Philosophy'}
A.M Vinogradov's (and J. Nestruev's) algebraic approach to observability on smooth manifolds can be extended to describe mechanical systems which possess inner structure, modelled by the differential geometry of fiber and vector bundles, possibly supplied with fiber-wise structures. 

In particular, in \cite{VinFat} the authors propose a simple yet elegant reinterpretation of the notion of a vector bundle in different terms $-$ what they call a \emph{fat manifold} i.e. as a 
 manifold of \emph{fat points} defined to be its fibers. 
 
 This terminology serves to emphasize we are dealing with points that have a non-trivial internal structure (in the case of a vector bundle, it is a linear space).
If single points of a fiber can be observed and consequently distinguished from one another with classical means, then they form a smooth manifold \cite[{\normalfont\S 10.1}]{Nestruev2}. 
Then, the point of view that the natural framework for the theory of linear connections as differential calculus in the category of (graded) modules over a (graded) commutative algebra combines naturally with the idea to treat a vector bundle as a ‘fat’ manifold composed of its ‘fat’ points.

This perspective puts the notions of fiber-wise linear structures on a bundle (e.g. vector fields) on the same footing as ordinary differential geometric structures on manifolds.
Actually, the literature on the subject of fiber-wise linear calculus beyond those operators of first-order is rather scarce, with the definition of a fiber-wise linear differential operator of order $\geq 1$ appearing only recently \cite{VitaglianoFWLDO}.

On the other hand, this formalism allows us to obtain natural differential complexes associated to not only vector bundles, but Lie algebroids, Jacobi algebroids and other interesting objects, that were historically introduced and studied independently and through external techniques and methods.
\vspace{1mm}

It is reasonable to expect that such topics may be treated in a unified point of view and the language of linear connections should form a basis of it. 
Indeed, favourable evidence was provided  within the formalism of diolic calculus \cite{Diole1}, where an analogue of differential forms as skew-symmetric multilinear functions on Der-operators gave an alternative characterization of Der-complexes \cite{Rub01}, which encode interesting  cohomological invariants
of projective modules, for instance Atiyah classes, Chern classes and Chern-Simons invariants. 

Moreover,
several other topics, introduced and studied independently appear to have a common origin within the language of diolic differential calculus. 
Indeed, Lie algebra representations are in bijective correspondence with the the datum of a diolic Lie algebra, Courant-Dorfman algebras correspond to diolic Poisson vertex algebras, and Lie algebroids appear as degree $-1$ Poisson
structures on algebras of dioles\footnote{Alias, \emph{fat Poisson manifolds} in the sense of \cite{VinFat}.}.

\vspace{1mm}

The fat formalism provides only a naive and geometrically demonstrative approach to a more sophisticated algebraic one. 
In fact, it should be considered as an introduction to a more satisfactory one based on the theory of dioles\footnote{Indeed, a general algebraic counterpart for fat manifolds is a pair $(A,P)$ composed of a commutative, unital algebra $A$ and a module $P$ over it, while a diole is a delicate combination of these structures $\EuScript{A}=A\oplus P.$}, trioles and ultimately, $n$-ole algebras that we sketch briefly in this work.

\subsection{Organization of the paper}
\label{ssec: Organization of the Paper}
We discuss a general algebraic formalism for calculus in vector bundles which are supplied with inner structures (Definition \ref{Definition: Inner Structure}), paying special attention to the case when this structure is a vector-valued fiber metric. 
\vspace{.5mm}

In this situation and even more general ones, there are natural group actions acting on the vector-valued forms and the algebraic invariants of these forms under such actions provide geometric invariants.  However, before one should hope to use these invariants, their algebraic structure must be known and a \emph{conceptual} origin and language should be established.
\vspace{.5mm}

The purpose of this paper is to discuss such a general algebraic framework, describe aspects of the algebraic approach to inner structures and extrapolate key \emph{graded} algebraic notions. 

We will find that, without the need to pass to more complicated or external situations, one may deduce an invariant differential calculus of inner structures and their related symmetries by simply finding the `right' context over which to elaborate the functors of differential calculus.

More exactly, we decipher a conceptual meaning of the term \emph{gauge structure}\footnote{This terminology is borrowed from \cite{VinFat}.} in differential geometry in the wider context of differential calculus over \emph{graded commutative} algebras.

For instance, in Theorem \ref{Theorem: Atiyah-Sequence Along Maps} we find an Atiyah-like sequence for operators along maps which generalize to describe higher order differential operators along maps (see Lemma \ref{Theorem: k-th order diolic differential operators along map}). Similarly, we find an analogous sequence for operators which act as symmetries for a vector-valued fiber metric and in Theorem \ref{Theorem: Triolic Differential Operators} describe a class of differential operators which may be suitably interpreted as generalized symmetries for the fiber metric.

\section{Conventions, Notation and Terminologies}
We will freely use the functors of differential calculus \cite{VinLogic},\cite{KraVerb},\cite{Ver} straightforwardly adapted to the setting of arbitrary graded commutative algebras and modules over them. We closely follow the conventions set up in \cite[{\normalfont\S 1.3.1}]{Diole1}, and thus recall only the most basic ingredients here.

\subsection{Differential Calculus over Graded Commutative Algebras}
\label{ssec: DCGCA}
We will deal with categories of $\mathcal{G}$-graded  commutative unital algebras over a fixed ground field $\mathbf{k}$ of characteristic zero, or more generally a graded ring $\mathcal{R}$ over $\mathbf{k}.$ For such a $\mathcal{G}$-graded commutative algebra $\mathcal{A}$, the associated category of $\mathcal{G}$-graded modules over it is denoted as $\mathbf{Mod}^{\mathcal{G}}(\mathcal{A}).$
\vspace{1mm}

The operator of multiplication of elements of an $\mathcal{A}$-module $\mathcal{P}$ by a homogeneous element $a\in \mathcal{A}$ is denoted $a_{\mathcal{P}}$. Then for 
$\mathcal{P}=\bigoplus_{g\in \mathcal{G}}P_g,\mathcal{Q}=\bigoplus_{g\in \mathcal{G}}Q_g\in\mathbf{Mod}^{\mathcal{G}}(\mathcal{A}),$ with $P_g,Q_g\in\mathbf{Mod}(A_0)$, we say that a morphism $\varphi:\mathcal{P}\rightarrow \mathcal{Q}$ is of \emph{degree} $g$ if it is $\mathbf{k}$-linear with $\varphi(P_h)\subseteq Q_{g+h}.$
We always understand the compositions $\varphi\circ a_{\mathcal{P}}$ and $a_{\mathcal{Q}}\circ \varphi$ for $a\in \mathcal{A}$ as maps $\mathcal{P}\rightarrow \mathcal{Q}:$
$$\varphi\circ a_{\mathcal{P}}(p):=\varphi(ap),\hspace{10mm} a_{\mathcal{Q}}\circ \varphi(p):=a\varphi(p),\hspace{2mm} p\in \mathcal{P}.$$
Let $\mathrm{Hom}^{\mathcal{G}}(\mathcal{P},\mathcal{Q}):=\bigoplus_g\mathrm{Hom}_{\mathcal{R}}^g(\mathcal{P},\mathcal{Q}),$ and $\mathrm{Hom}_{\mathcal{R}}^g(\mathcal{P},\mathcal{Q})=\{\varphi:\mathcal{P}\rightarrow \mathcal{Q}|\varphi(\mathcal{P}_{h})\subseteq \mathcal{Q}_{h+g},h\in G\}.$
In $\text{Hom}_{\mathcal{R}}\big(\mathcal{P},\mathcal{Q}\big),$ we have two graded $\mathcal{A}$-module structures (left and right):
$$
a^<\varphi:=a_{\mathcal{Q}}\circ\varphi,\hspace{5mm}
a^>\varphi:=(-1)^{a\cdot \varphi}\varphi\circ a_{\mathcal{P}}.$$
For homogeneous $a\in\mathcal{A}$ and  $\varphi\in\mathrm{Hom}_{\mathcal{R}}(\mathcal{P},\mathcal{Q}),$ 
\begin{equation}
\delta_a\varphi:=\big[a,\varphi]=a^<\varphi-a^>\varphi,
\end{equation}
and for $a_0,a_1,\ldots,a_k\in \mathcal{A}$, we put $\delta_{a_0,\ldots,a_k}:=\delta_{a_0}\circ\ldots\circ \delta_{a_k}.$
Notice that $\delta_{a,b}=(-1)^{ab}\delta_{b,a}$ and $\delta_{ab,c}=a\delta_{b,c}+(-1)^{ab}b\delta_{a,c}$,

\begin{definition}
\label{Def: Graded DO}
A map
$\Delta\in\mathrm{Hom}_{\mathcal{R}}(\mathcal{P},\mathcal{Q})$ is a \emph{graded differential operator} of order $\leq k$ if $\delta_{a_0,\ldots,a_k}(\Delta)=0,$ for all $a_0,\ldots,a_k\in \mathcal{A}$. It is homogeneous of \emph{degree} $g\in \mathcal{G}$ if $\Delta\in \mathrm{Hom}^g(\mathcal{P},\mathcal{Q}).$
\end{definition}
The set of order $\leq k$ graded differential operators is denoted by 
$\mathrm{Diff}_k(\mathcal{P},\mathcal{Q})$ and carries the obvious natural left (resp. right) $\mathcal{A}$-module structures denoted by $\mathrm{Diff}^<$ (resp. $\mathrm{Diff}^>$). 
Just as in the ungraded situation, we have $\mathrm{Diff}_0(\mathcal{P},\mathcal{Q}):=\mathrm{Hom}_{\mathcal{A}}(\mathcal{P},\mathcal{Q})$ and the order filtration of graded $\mathcal{A}$-bimodules
$$\mathrm{Diff}_0(\mathcal{P},\mathcal{Q})\subset \mathrm{Diff}_1(\mathcal{P},\mathcal{Q})\subset \ldots \subset \mathrm{Diff}_{k}(\mathcal{P},\mathcal{Q})\subset \mathrm{Diff}_{k+1}(\mathcal{P},\mathcal{Q})\subset\ldots$$
with $\mathrm{Diff}(\mathcal{P},\mathcal{Q}):=\bigcup_{k\geq 0}\mathrm{Diff}_k(\mathcal{P},\mathcal{Q}).$
Moreover, $\mathrm{Diff}(\mathcal{P},\mathcal{P})$ is a $\mathcal{G}$-graded left (noncommutative) $\mathcal{A}$-algebra with respect to the operation $a\cdot (\Delta_g\circ\nabla_h):=\big(a\cdot \Delta_g)\circ \nabla_h,$ and $\mathrm{Diff}(\mathcal{P},\mathcal{Q})$ is both a $\mathcal{G}$-graded right $\mathrm{Diff}(\mathcal{P},\mathcal{P})$ and left $\mathrm{Diff}(\mathcal{Q},\mathcal{Q})$-modules.
The interested reader is referred to \cite{Nestruev2} §9.66, §9.67] for details in the ungraded situation.

\begin{remark}[Notation]
If $\Delta:\mathcal{P}\rightarrow \mathcal{Q}$ is a differential operator of order $\leq k$ and grading $g$, denote the restriction to the $h$-th homogeneous component by $\Delta^{\mathcal{P}_h}:=\Delta|_{\mathcal{P}_h}:\mathcal{P}_h\rightarrow \mathcal{Q}_h[g]=\mathcal{Q}_{h+g}.$ Obviously, $\Delta^{\mathcal{P}_h}$ is a differential operator of order $\leq k$ in the category of $\mathcal{A}_0$-modules.
\end{remark}

The importance of Definition \ref{Def: Graded DO} is exemplified in that it may be used to generate other functors of calculus. For instance, \emph{graded derivations} are obtained by putting
$D_1(-)_{\mathcal{G}}:\mathbf{Mod}^{\mathcal{G}}(\mathcal{A})\rightarrow \mathbf{Mod}^{\mathcal{G}}(\mathcal{A}), \hspace{1mm} \mathcal{P}\mapsto D(\mathcal{P}):=\{\Delta\in\mathrm{Diff}_1^<(\mathcal{P})|\Delta(1_{\mathcal{A}})=0\}.$ 
Standard functors of differential calculus can be obtained in a similar fashion, but we do not discuss their definitions here.

The expected representability results for such functors hold in the graded situation.
For example, 
the functor $D_1^{\mathcal{G}}$ is represented by a $\mathcal{G}$-graded module of differential $1$-forms $\Omega^1(\mathcal{A})$, equipped with a unique degree zero graded derivation $\delta:\mathcal{A}\rightarrow \Omega^1(\mathcal{A}),$ satisfying its usual universal property \cite{Nestruev2}. 
An analogous universal property holds for functors $\mathrm{Diff}_k$ and its representing module of $k$-jets $\mathcal{J}^k.$

\subsection{Derivations in Vector Bundles}
\label{sssec: Der-operator section}
Given an $A$-module $P$, there is an $A$-submodule $\mathrm{Der}(P)\subset\mathrm{Diff}_1(P,P),$ whose elements are first order differential operators with scalar symbol. 

It is known that they provide the appropriate algebraic analog of derivations in vector bundles i.e. those $\mathbb{R}$-linear maps
$\nabla:\Gamma(E)\rightarrow \Gamma(E)$ satisfying $\Delta(fe)=\sigma_{\Delta}(f)e+f\Delta(e)$ for a unique vector field $\sigma_{\Delta}\in D(M).$
Such operators with zero symbol are simply  endomorphisms and so there is an exact sequence of Lie algebras,
\begin{equation}
    \label{eqn: Atiyah sequence}
   0\rightarrow \mathrm{End}(E)\rightarrow \mathrm{Der}(E)\rightarrow TM\rightarrow 0.
\end{equation}
Sequences of the above form are called \emph{Atiyah sequence} of the bundle $E$, originally introduced to study the existence of flat connections in the bundle \cite{Ati}.
\begin{remark}
It is worth mentioning that recently, a class of Atiyah-like sequences  have appeared in another context, via the description of functors of graded differential calculus in diole algebras \cite{Diole1}.
\end{remark}

Der-operators are also known to describe covariant derivatives associated to linear connections. This is clear if we understand a linear connection in $P$ as an $A$-module homomorphism $\nabla:D(A)\rightarrow \mathrm{Der}(P)$ satisfying $\nabla(fX+gY)=f\nabla(X)+g\nabla(Y)$ for all $f,g,\in A$ and $X,Y\in D(A)$ such that the following so-called \emph{Der-Leibniz rule} 
\begin{equation}
    \label{eqn: Der-Leibniz Rule}
    \nabla_X(f\cdot p)=X(f)\cdot p+f\cdot \nabla_X(p),
\end{equation}
holds for all $f\in A,p\in P.$ 
 Obviously, the connection $\nabla$ is flat if $\nabla$ is a Lie-algebra homomorphism.

\section{Inner Structures in Projective Modules}
The language for inner structures in a vector bundle that we utilize is formulated within the framework of multi-linear algebra related to our projective $A$-module $P.$

Specifically, consider the following tensor algebra of $P$,
$$\mathcal{T}(P):=\bigoplus_{p,q\geq 0}P_q^p,\hspace{4mm}\text{with } \hspace{1mm} P_q^p:=P^{\otimes_A p}\otimes_A (P^*)^{\otimes_A q},$$
with $P^*$ the $A$-linear dual.
It is clear that $\mathcal{T}(P)$ is a graded commutative algebra with multiplication defined in the obvious way, and is equipped with its usual $\mathrm{Aut}(P)$-action. We will also write $\mathrm{GL}(m;\mathbf{k})$ for this group.

Namely, if $\varphi\in \mathrm{Aut}(P)$ one extends it to an action on $\mathcal{T}(P)$ by setting $\varphi(\mathcal{P}_1\otimes\mathcal{P}_2):=\varphi(\mathcal{P}_1)\otimes\varphi(\mathcal{P}_2)$ for decomposable tensors $\mathcal{P}_i\in \mathcal{T}(P).$
More generally, given another projective $A$-module $Q$ we consider
$\mathcal{T}(P)\otimes\mathcal{T}(Q):=\bigoplus_{(p,q)}\bigoplus_{(r,s)}P_q^p\otimes_A Q_s^r$, where we have put
$$P_q^p\otimes_A Q_s^r=P^{\otimes_A p}\otimes_A (P^*)^{\otimes_A q}\otimes_A Q^{\otimes_A r}\otimes_A(Q^*)^{\otimes_A s}.$$

If $\mathrm{dim}(P)=m_P$ and $\mathrm{dim}(Q)=m_Q,$ there is a natural action by $\mathrm{GL}(m_P;\mathbf{k})\times\mathrm{GL}(m_Q;\mathbf{k}).$

\begin{example}
Consider a vector-valued bilinear form i.e.  an $A$-linear map $g:P\otimes P\rightarrow Q.$ 
If $P,Q$ are projective and finitely generated, the $A$-module of such forms $\mathrm{Bil}(P,Q)$ is identified with $P^*\otimes P^*\otimes Q.$ 
Viewed as an $A$-bilinear map $g:P\times P\rightarrow Q,$ the $\mathrm{Aut}(P)\times\mathrm{Aut}(Q)$-action is natural in the sense that the diagram
\begin{equation*}
\begin{tikzcd}
P\times P\arrow[d,"\varphi\times\varphi"] \arrow[r,"G"]&  Q\arrow[d,"\psi"]
\\
P\times P\arrow[r," \widetilde{G}"] &Q
\end{tikzcd}
\end{equation*}

is commutative.
\end{example}

\begin{definition}
\label{Definition: Inner Structure}
A element $\Xi\in \mathcal{T}(P)\otimes\mathcal{T}(Q)$ of homogeneous degree $(p,q)\times (r,s)$ which is invariant under the natural $\mathrm{Aut}(P)\times\mathrm{Aut}(Q)$-action is said to be an \emph{inner structure of type} $(p,q)\times (r,s).$
\end{definition}
More exactly, $\Xi$ represents the equivalence class of elements of $P_q^p\otimes_A Q_s^r$ under the above action i.e. lies in the orbit.

\begin{example}
\label{Example: Inner Vector-Valued Form}
To say that a vector valued bilinear form
$g\in P_2^0\otimes Q_0^1$ is an inner structure means we are considering all $g$ and $\widetilde{g}$ to be equivalent if $\widetilde{g}$ is of the form $\widetilde{g}=\psi\circ g\big(\varphi^{-1}\otimes\varphi^{-1}\big)$ for $\Upsilon=(\varphi,\psi)\in \mathrm{Aut}(P)\times\mathrm{Aut}(Q).$
\end{example}

\subsection{Symmetry Algebras and Gauge Structures}
To each inner structure one may attribute its symmetry algebra. Specifically, given a type $(p,q)$ inner structure $\Xi\in P_q^p,$ we may consider its \emph{Lie algebra of infinitesimal symmetries}
\begin{equation}
    \label{eqn: Inner Structure Symmetry Lie Algebra}
    \mathfrak{o}(P;\Xi):=\big\{\varphi\in \mathrm{End}_A(P)\mid \sum_{i=1}^q \Xi\big(p_1,\ldots,\varphi(p_i),\ldots,p_q)=0\big\}.
\end{equation}

This $A$-module is the infinitesimal counterpart of the space of \emph{symmetries} of $\Xi$, denoted by $\mathrm{O}(P;\Xi):=\{\varphi\in \mathrm{Aut}(P)\mid \Xi\big(\varphi(p_1),\varphi(p_2),\ldots,\varphi(p_q)\big)=\Xi(p_1,p_2,\ldots,p_q)\}.$ 
\vspace{2mm}

\subsubsection{Geometric Picture}
Consider a vector bundle $E\rightarrow M$. Then a symmetry is a regular vector bundle morphism $\overline{f}$ covering the identity, such that 
$\Xi\big(\overline{f}^*(p_1),\overline{f}^*(p_2),\ldots,\overline{f}^*(p_q)\big)=\Xi(p_1,\ldots,p_q),$ where $\overline{f}^*$ is the universal homomorphism associated with a regular vector bundle morphism (see for instance, \cite[{\normalfont\S 0.3.13}]{VinFat}). 
It is easy to see then that $\mathfrak{o}(P;\Xi)$ is indeed the appropriate infinitesimal counterpart of $\mathrm{O}(P;\Xi).$ Indeed, take a one-parameter family of such vector bundle morphisms $\{\overline{f}_t\}_{t\in \mathbb{R}}$ with $\overline{f}_0=\mathrm{id}$ and suppose that each $\overline{f}_t$ is a symmetry of $\Xi$ so that
\begin{equation}
    \label{eqn: Symmetry to Inf Symmetry Relation}
\Xi\big(\overline{f}_t^*(p_1),\ldots,\overline{f}_t^*(p_q)\big)=\Xi(p_1,\ldots,p_q),\hspace{3mm} p_1,\ldots, p_q\in \Gamma(E).
\end{equation}
By formally taking the derivative $\frac{d}{dt}$ of relation (\ref{eqn: Symmetry to Inf Symmetry Relation}), we obtain that
$$\Xi\big(\frac{d}{dt}\mid_{t=0}\overline{f}_t^*(p_1),p_2,\ldots,p_q\big)+\ldots+\Xi\big(p_1,p_2,\ldots,\frac{d}{dt}\mid_{t=0}\overline{f}_t^*(p_q)\big)=0.$$
By setting $\varphi:=\frac{d}{dt}\mid_{t=0}\overline{f}_t^*,$ we recover the defining relation for the infinitesimal symmetry algebra (\ref{eqn: Inner Structure Symmetry Lie Algebra}).
\vspace{1mm}

\subsubsection{Examples.}
Consider the situation when $P$ is an $A$-module supplied with a bilinear form $b:P\times P\rightarrow A.$
The $b$-\emph{orthogonal group} of $P$ is 
\begin{equation}
\label{orthogonalgroup}
\mathrm{O}(P,b):=\{\varphi\in\mathrm{Aut}(P)|b\big(\varphi(p_1),\varphi(p_2)\big)=b(p_1,p_2), p_1,p_2\in P\}.
\end{equation}
Moreover, the sub-module of infinitesimal symmetries is
$$\mathrm{o}(P,b):=\{\varphi\in\mathrm{End}(P)|b\big(\varphi(p_1),p_2\big)+b\big(p_1,\varphi(p_2)\big)=0\}\subset \mathrm{End}(P).$$
This is not only an $A$-sub-module, but a Lie sub-algebra as well.

By varying the type of inner structure, many important symmetry algebras can be found. For instance given $\psi\in\mathrm{End}(P)$, the corresponding module of \emph{symmetries of} $\psi$ is
\begin{equation*}
    \mathrm{GL}(P,\psi):=\{\Phi\in\mathrm{Aut}(P)|\big[\Phi,\psi\big]=0\}.
\end{equation*}
All endomorphisms of $P$ commuting with $\psi$ constitute a Lie sub-algebra $\mathrm{gl}(P,\psi)$ of $\mathrm{End}(P).$

Another important example arises by considering an (inner) complex structure $J$ i.e. an endomorphism $J:P\rightarrow P$ such that $J^2=-\mathrm{id}_P$. 

\subsection{Gauge Structures}
A linear connection $\nabla$ in $P$ is said to \emph{preserve} an inner structure $\Xi\in \mathcal{T}(P)$ if $d_{\nabla^{\mathcal{T}(P)}}(\Xi)=0,$ where $\nabla^{\mathcal{T}}$ is the induced connection in $\mathcal{T}(P)$ and $d_{\nabla}$ is the usual differential associated to a flat connection. An inner structure in an $A$-module $P$ that admits a preserving-it
linear connection is called a \emph{gauge structure}. We will require a slightly more general definition.

\begin{definition}
\label{General Gauge Structure Definition}
Suppose that $P,Q$ are endowed with flat linear connections $\nabla,\Box$, respectively with shared scalar symbol.
Then $\Xi$ is said to be a \emph{gauge structure of type $(p,q)\times (r,s)$} if $d_{\nabla^{\mathcal{T}}}(\Xi)=0$.
\end{definition}
That is, $\Xi$ is a gauge structure if
\begin{equation}
    \label{eqn: Gauge Structure Condition}
    \nabla_X^{\mathcal{T}}(\Xi)=0,\hspace{3mm} \text{ for each }\hspace{1.5mm} X\in D(A).
\end{equation}

Let us describe gauge bilinear forms i.e. gauge structures of type $(0,2).$
\begin{example}
Let $P,\nabla$ be as above and consider a bilinear form $b:P\times P\rightarrow A.$ 

The Der-Leibniz rule for the induced connection $\big(\nabla\otimes \nabla\big)^{*},$ gives that
\begin{equation}
    \label{bilinearformconnection}
\nabla_X^{\text{Bil}}(b)(p_1,p_2)=X\big(b(p_1,p_2)\big)-b\big(\nabla_X(p_1),p_2\big)-b\big(p_1,\nabla_X(p_2)\big).
\end{equation}
That $\nabla$ preserves $b$ means
$\nabla_X^{\mathrm{Bil}}(b)=0,$ for each $X\in D(A)$ which reads as 
$$X\big(b(p_1,p_2)\big)=b\big(\nabla_X(p_1),p_2\big)+b\big(p_1,\nabla_X(p_2)\big).$$
\end{example}
\vspace{.5mm}

\subsubsection{Vector-Valued Forms}
\label{sssec: Vector-Valued Forms}
Consider two projective $A$-modules equipped with linear connections $(P,\nabla)$ and $(Q,\Box)$ with shared scalar symbol $\sigma(\nabla)=\sigma(\Box)=X\in D(A)$ and fix a vector-valued bilinear form $g:P\otimes P\rightarrow Q$ (an inner structure of type $(0,2)\times (1,0)$).

There is an induced linear connection in the $A$-module of such forms $\mathrm{Bil}(P;Q)$ defined by:
$$\nabla^{\mathrm{Bil}(P;Q)}:D(A)\rightarrow \mathrm{Der}\big(\mathrm{Bil}(P;Q)\big),\hspace{2mm} X\mapsto \nabla_X^{\mathrm{Bil}(P;Q)},$$
which upon evaluation on $g\in \mathrm{Bil}(P;Q)$ is given by
$$\nabla_X^{\mathrm{Bil}(P;Q)}(g)(p_1\otimes p_2)=\Box_X\big(g(p_1\otimes p_2)\big)-g\big(\nabla_X(p_1)\otimes p_2\big)-g\big(p_1\otimes \nabla_X(p_2)\big),$$
for all $p_1,p_2\in P,X\in D(A).$

The following result then clarifies when $g$ is a gauge structure.
\begin{lemma}
A vector-valued bilinear form $g\in \mathrm{Bil}(P;Q)$ is a gauge structure if and only if
$\Box_X\big(g(p_1\otimes p_2)\big)=g\big(\nabla_X(p_1)\otimes p_2\big)+g\big(p_1\otimes \nabla_X(p_2)\big),$
for all $p_1,p_2\in P,X\in D(A).$
\end{lemma}
In this case, one says that the pair $(\nabla,\Box)$ of connections is $g$-preserving if
$\Box_X\big(g(p_1,p_2)\big)-g\big(\nabla_X(p_1),p_2\big)-g\big(p_1,\nabla_X(p_2)\big)=0,$
for every $X\in D(A).$

\begin{lemma}
Any connection in $P$ which preserves a bilinear form $b:P\times P\rightarrow A$ is equivalently given by a $\tilde{g}:P\times P\rightarrow Q$-preserving pair $(\nabla,D_X)$ in the trivial bundle corresponding to the $A$-module $Q,$ where $D_X:=X$ is the standard trivial connection for each $X\in D(A).$
\end{lemma}

A systematic study of gauge structures in arbitrary modules, including their classification up to a suitable notion of equivalence is possible. For instance, given a vector bundle $E$, a linear connection $\nabla$, and an inner structure $\Xi\in \Gamma(E^{\otimes p}\otimes (E^*)^{\otimes q}),$ subject to natural compatibilities between $\nabla$ and $\Xi$, it is possible to make precise when two such pairs are \emph{equivalent} and by attributing to such objects certain cohomological invariants, one may classify these structures (see §§\ref{ssec: Future Directions} below for future perspectives).
Treating these topics lies outside the scope of this paper. 
\vspace{1mm}

The remainder of this paper is devoted to showing how the information concerning simple classes of gauge structures and their symmetry algebras is contained in the description of functors of differential calculus over special graded commutative algebras.

\section{An `Ecosystem' of Graded Algebras and their Differential Calculus}
We begin by introducing a general species of graded commutative algebras and specializing our discussion to the simple case of what we call \emph{triole algebras}. Roughly speaking this formalism corresponds to calculus in the presence of a gauge structure of the type $(0,2)\times (1,0)$.
This definition also includes the diole algebras introduced in \cite{Diole1}.

\subsection{$n$-olic Algebras}
Fix an integer $n\in\mathbb{Z}_{>0}.$
\begin{definition}
\label{n-ole algebra definition}
An \emph{n-ole algebra} is a graded commutative algebra $\EuScript{N}=\bigoplus_{i=0}^n N_i$ with the following properties:
\begin{itemize}
    \item $A:=N_0$ is a unital commutative $\mathbf{k}$-algebra;
    \item $N_i\in \mathrm{Mod}(A)$ for $i=1,\ldots,n$ (i.e. there are multiplications $A\cdot N_i\subseteq N_i$),
    \item There are a family of $A$-bilinear maps $\beta_j^k:N_k\otimes_A N_j\rightarrow N_{k+j}$ for $k=1,\ldots,N$ and $j=k,\ldots,(n-k),$
\end{itemize}
subject to the further conditions that $N_i\cdot N_j:=\emptyset$ for all $i+j>n$ and with the property that $N_n\cdot N_n:=\emptyset.$
\end{definition}
The algebra $\EuScript{N}$ is associative, so we have that the standard relation
$\mu\circ \big(\mathbf{1}\otimes \mu\big)=\mu\circ \big(\mu\otimes \mathbf{1}\big)$
holds, where $\mu$ is the algebra multiplication. 
On homogeneous components this is an associativity relation between the family of $A$-bilinear maps $\beta_j^k$ given by
\begin{equation}
    \label{eqn: Associativitiy for N-ole algebras}
    \beta_{j+k}^i\circ \big(\mathbf{1}_{N_i}\otimes \beta_k^j\big)=\beta_k^{i+j}\circ \big(\beta_j^i\otimes \mathbf{1}_{N_k}\big),
    \end{equation}
where $\mathbf{1}:\EuScript{N}\rightarrow \EuScript{N}$ is the identity map and $\mathbf{1}_{N_k}:=\mathbf{1}|_{N_k},$ with $i+j+k\leq n.$
\\

\subsubsection{Modules}
Two interesting notions of a module over an $n$-ole algebra can be found. 

The first describes a (left) $\EuScript{N}$-module as a $\mathbb{Z}$-graded module $\EuScript{W}=\bigoplus_{i\geq 0}\EuScript{W}_i$ equipped with an action $\nu\colon \EuScript{N}\otimes\EuScript{W}\rightarrow \EuScript{W}$ such that the diagrams
\[
\begin{tikzcd}
\EuScript{N}\otimes\EuScript{N}\otimes\EuScript{W}\arrow[d,"\mu\otimes \EuScript{W}"] \arrow[r, "\EuScript{N}\otimes \nu"] & \EuScript{N}\otimes\EuScript{W}\arrow[d,"\nu"]
\\
\EuScript{N}\otimes\EuScript{W}\arrow[r,"\nu"]&\EuScript{W}
\end{tikzcd}\hspace{50pt}
\begin{tikzcd}
1\otimes\EuScript{W}\arrow[d,"i\otimes\EuScript{W}"] \arrow[r,"\cong"] & \EuScript{W}
\\
\EuScript{N}\otimes\EuScript{W}\arrow[ur,"\nu"]
\end{tikzcd}
\]
commute, where $i$ is the unit map for the algebra $\EuScript{N}$ and $\mu$ is the algebra multiplication.
Morphisms are defined simply as morphisms of graded modules over a graded algebra.

We also have the notion of a \emph{strict} morphism of left $\EuScript{N}$-modules as those $\mathbf{k}$-linear maps $\phi:\EuScript{W}\rightarrow\EuScript{V}$ such that $\phi(W_g)\subset V_g$ for each $g\in\mathcal{G}.$ This class of morphisms gives rise to a category $\mathrm{Mod}_{\mathrm{str}}^{\mathcal{G}}(\EuScript{N})$. More generally, we have a category $\mathrm{Mod}^{\mathcal{G}}(\EuScript{N})$ whose space of morphisms is given by $\mathrm{Hom}_{\EuScript{N}}^{\mathcal{G}}(\EuScript{W},\EuScript{V})=\bigoplus_{g\in G}\mathrm{Hom}_{\EuScript{N}}^g\big(\EuScript{W},\EuScript{V}\big)$, 
with those in the summand of degree $g$ given by those $\EuScript{N}$-linear maps $\varphi:\EuScript{W}\rightarrow \EuScript{V}$ such that $\varphi(W_j)\subseteq V_{j+g}.$ Note that $\EuScript{N}$-linearity here refers to the obvious graded linearity, $\varphi(a\cdot w)=(-1)^{a\cdot \varphi}a\cdot \varphi(w)$ for each $a\in \EuScript{N}$ and $w\in \EuScript{W}.$

The second interesting class of modules form a sub category in $\mathrm{Mod}^{\mathcal{G}}(\EuScript{N}).$
\begin{definition}
\label{TruncatedTriolicModuleDefinition}
A \emph{truncated $\EuScript{N}$-module} is a $\mathbb{Z}$-graded $\EuScript{N}$-module $\EuScript{W}:=\bigoplus_{i=1}^nW_i$ with $W_i:=\emptyset$ for all $i>n,$ together with the datum:
$$(\EuScript{W},\{\lambda_k^j\}_{k=1\ldots,n,j=1,\ldots,n-k})=(\EuScript{W},\lambda_A^{i},\lambda_1^{j},\ldots,\lambda_{k}^{j},\ldots),$$
for each $k$ we have that $\lambda_k^j:N_k\otimes_A W_j\rightarrow W_{k+j}$ for $j=1,\ldots,n-k$ and where $\lambda_A^i$ are the $A$-module structures on $W_i$ for all $i=1,\ldots,n$.
These are
subject to the compatibility condition
\begin{equation}
\label{eqn:truntrimod}
\nu_{\EuScript{N}_i}^{(j+k)}\circ\big(\mathbf{1}_{N_i}\otimes \nu_{\EuScript{N}_j}^{(k)}\big)=\nu_{\EuScript{N}_{i+j}}^{(k)}\circ \mu_{i,j}\circ \alpha_{ijk},
\end{equation}
where 
$\alpha_{ijk}:\EuScript{N}_i\otimes (\EuScript{N}_j\otimes \EuScript{W}_k)\rightarrow (\EuScript{N}_i\otimes\EuScript{N}_j)\otimes\EuScript{W}_k,$ is the associator isomorphism, $\mu_{i,j}:\EuScript{N}_i\otimes\EuScript{N}_j\rightarrow \EuScript{N}_{i+j}$ is the multiplication.
\end{definition}

Truncated modules over $n$-olic algebras for $n>1$ will generally carry a family of compatibilities, but note that in the simplest situation of a diole algebra $\EuScript{A}=A\oplus P$ there are no compatibility relations and to define a truncated diole module $\EuScript{P}=P_0\oplus P_1$ it suffices to specify an $A$-bilinear map  $\beta:P\otimes_A P_0\rightarrow P_1.$
\\

\subsubsection{Example: Diole Algebras}
\label{sssec: Diole Algebras}
Recall the definition of a diole algebra and of a truncated diole module $(\EuScript{P},\beta)$ from \cite{Diole1}.
We call $\EuScript{A}=(A,P)$ geometric if  $A$ is smooth and $P$ is a geometric $A$-module. Of course, the prototype diole algebra is that associated to a vector bundle $\pi:E\rightarrow M$ i.e. $\EuScript{A}_{\pi}:=C^{\infty}(M)\oplus \Gamma(\pi).$

\begin{example}
For a diole algebra $\EuScript{A}=A\oplus P$ 
the trivial pair $(A\oplus P,\iota)$ where $\iota:P\otimes_A A\cong P$ is a truncated diolic module.
\end{example}

\begin{example}
\label{Truncated diole module example}
The pair $\big(\Omega^1(A),\mathcal{J}^1(A)\big)$ is a truncated diole module. with the map $\beta:P\otimes \Omega^1(A)\rightarrow \mathcal{J}^1(P)$ defined by $\beta(p\otimes da):=j_1(ap)-aj_1(p),$ with $j_1$ the universal $1$-jet operator.
Note further that this map is $A$-linear $\beta(bp\otimes da)=b\beta(p\otimes da)$ and satisfies 
$\beta\big(p\otimes d(ab)\big)=\beta(p\otimes adb)+\beta(p\otimes bda).$
\end{example}

\begin{example}
If $\EuScript{A}$ is a geometric diole alebra, then the pair $\Omega^1(A)\oplus \Omega^1(P)$ is a diolic module with structure map $\beta:P\otimes_A \Omega^1(A)\rightarrow \Omega^1(P)$ induced by the isomorphism of $A$-modules $\Omega^1(P)=\Omega^1(A)\otimes P.$
\end{example}

\begin{example}
The pair $\mathcal{J}^k(\EuScript{A}):=\mathcal{J}^k(A)\oplus \mathcal{J}^k(P)$ is a diolic module under the isomorphism $\mathcal{J}^k(P)\cong \mathcal{J}^k(A)\otimes_A P$ identifying
$j_k^P(p)$ with $j_k^A(1_A)\otimes_A p.$
\end{example}

\begin{example}
\label{Example: Diole Along Map}
Let $\varphi:A\rightarrow B$ be a morphism of commutative algebras and let $Q$ be a left $B$-module, viewed as an $A$-module via $\varphi$. When equipped with such a module structure, write $Q^{<\varphi}.$
To any $A$-module homomorphism $\overline{\varphi}:P\rightarrow Q$ there is an associated diole module
$\EuScript{P}_{\overline{\varphi}}:=B\oplus Q^{<\varphi}$ with $\beta_{\overline{\varphi}}:B\otimes_A P\rightarrow Q$ given by universal property of scalar extensions.
\end{example}

\subsection{Triole Algebras}
\label{sssec: Triole Algebras}
A triole algebra is a $\mathbb{Z}$-graded commutative algebra  $\EuScript{T}:=\mathcal{T}_0\oplus \mathcal{T}_1\oplus \mathcal{T}_2$ such that $\mathcal{T}_0=A$ is a commutative unital $\mathbf{k}$-algebra, with $\mathcal{T}_1,\mathcal{T}_2$ are $A$-modules, such that $\mathcal{T}_i:=0,\forall i\neq 0,1,2$ and $\mathcal{T}_2\cdot \mathcal{T}_2:=0.$
Algebra multiplication is generated by an $A$-bilinear form $g:\mathcal{T}_1\otimes \mathcal{T}_1\rightarrow \mathcal{T}_2.$ It is said to be \emph{regular} if $g$ is non-degenerate.
\begin{remark}
Given a bilinear form $g\in\text{Bil}(P,Q),$ we have the corresponding adjoint morphism
$
    \mathfrak{g}:=g^{\sharp}:P\rightarrow \mathrm{Hom}_A(P,Q),\hspace{2mm} \mathfrak{g}(p_1)p_2:=g(p_1,p_2).$
We say that  
$g$ is regular or non-degenerate if $\mathfrak{g}$ is an isomorphism of $A$-modules.
\end{remark}

\begin{example}
A triole algebra can be associated to: a pair of vector bundle $\pi:E\rightarrow M$ and $\eta:F\rightarrow M$ with $E$ endowed with a $\Gamma(\eta)$-valued fiber metric, $g:\Gamma(E)\otimes_{C^{\infty}(M)}\Gamma(E)\rightarrow \Gamma(F),$ denoted by $\EuScript{T}_{\pi,\eta}:=C^{\infty}(M)\oplus (\Gamma(\pi),g)\oplus \Gamma(\eta).$ 
\end{example}

Here are several interesting examples of triole algebras arising in geometry.
\begin{example}
Let $L$ be a line bundle over a smooth manifold or algebraic variety $X$ with module of sections $Q$  and let $a_1,\ldots,a_n$ be global sections of the group $\mathbb{G}_m.$ Consider
$\mathscr{T}_Q^{\mathbb{G}_m}:=A\oplus \big(Q^{\oplus n},\ell(a_,\ldots,a_n)^Q\big)\oplus Q^{\otimes 2}/T^{> 2}Q,$
with bilinear form
$\ell(a_1,\ldots,a_n)^Q:Q^{\oplus n}\otimes_A Q^{\oplus _n}\rightarrow Q^{\otimes 2}/T^{>2}Q,$
defined by 
$(q_1,\ldots,q_n)\otimes (q_1',\ldots,q_n')\longmapsto \sum_{i=1}^na_iq_i\otimes q_i'.$
This defines triolic algebra over $A,$ precisely since the natural multiplication in $Q^{\otimes 2}$ inherited from the entire tensor algebra $T^{\bullet}Q,$ squares to zero in the quotient $Q^{\otimes 2}/T^{>2}Q.$
\end{example}
\begin{example}
Let $\varphi:B\hookrightarrow A$ be an embedded sub-algebra and suppose that $I\subseteq B$ is a two-sided ideal of $B$. Let $\iota:B\rightarrow A$ be an arbitrary algebra morphism.

There is a triole algebra 
$$\EuScript{T}:=B\oplus (I,g)\oplus A^{<\varphi},\hspace{2mm} g:I\otimes_{\mathbf{k}}I\rightarrow A^{<\varphi},\hspace{2mm} g:=\big(\iota \otimes \iota\big)|_{I\times I}\rightarrow A.$$

If $\iota$ is surjective, $\mathrm{Im}(g)$ is an ideal in $A.$
\end{example}
The above example can be specified to the case of a sub-manifold and a distribution.
\begin{example}
Let $i:S\hookrightarrow M$ be a sub-manifold and $D_S\subseteq TS$ a distribution. Denote $C_D^{\infty}(S)$ the functions on $S$  constant along leaves of $D$, and write $C_D^{\infty}(M):=\{f\in C^{\infty}(M)|i^*(f)\in C_D^{\infty}(S)\}.$ Let $\mathcal{I}_S$ be the ideal of $S$ i.e. $\mathcal{I}_S:=\{f\in C^{\infty}(M)|i^*(f)=0\}.$
Then 
$\EuScript{T}=C_D^{\infty}(M)\oplus \mathcal{I}_S\oplus C^{\infty}(M),$
arising as in the above example is a triole algebra by declairing the product of two functions on $M$ to be zero.
\end{example}

\section{Basic Functors of Differential Calculus in $n$-olic Algebras}
Let us describe the functors of graded derivations in the setting of $n$-olic algebras. We will then interpret the resulting operators in the case of diolic and triole algebras. 
As an application, we will extract two generalized Atiyah-like sequences associated to these operators.

A more systematic study of the functors of differential calculus in triole algebras is given in \cite{Triole1}, while the general situation for $n$-ole algebras is a subject of our future work.

\begin{lemma}
Let $\EuScript{N}$ be an $n$-olic algebra.
The module of graded derivations admits a homogeneous decomposition as 
$$D(\EuScript{N})=D(\EuScript{N})_0\oplus D(\EuScript{N})_1\oplus D(\EuScript{N})_2\oplus D(\EuScript{N})_3\oplus\ldots\oplus  D(\EuScript{N})_n,$$
which admit the following descriptions:
\begin{itemize}
\item Degree $0$: are $n$-tuples of operators
$X=\big(X^A,X^1,\ldots,X^n\big),$
where $X^A\in D(A)$, $X^i\in \mathrm{Der}(N_i)$ with $\sigma(X^i)=X^A$ for $i=1,\ldots, n$ which satisfy the symmetry relations
\begin{equation}
    \label{eqn: Deg Zero Symmetry Relation}
X\big(\beta_j^i(p_i,p_j)\big)=\beta_j^i\big(X^i(p_i),p_j\big)+\beta_j^i\big(p_i,X^j(p_j)\big);
\end{equation}
    \item Degree $1$: are tuples $X_1=(X_1^A,X^1,\ldots,X^{n-1}),$ where $X_1^A\in D(A,N_1)$ and $X^k\in \mathrm{Der}^{\beta_k^1}(N_k,N_{k+1})$ i.e. $X_1(a\cdot p_k)=\beta_k^1\big(X_1^A(a)\otimes p_k\big)+aX^k(p_k)$ for all $k=1,\ldots,n-1,$ and moreover satisfy the symmetry relation $X_1\big(\beta_j^i(p_i,p_j)\big)=\beta_j^{i+1}\big(X^i(p_i),p_j\big)+(-1)^i\beta_{j+1}^i\big(p_i,X^j(p_j)\big),$ for all $i+j\leq n;$
    
   \item $\ldots$

    \item Degree $\ell$: are $X_{\ell}=(X_{\ell}^A,X^1,\ldots,X^{n-\ell})$ where $X_{\ell}^A\in D(A,N_{\ell})$ and $X^k\in \mathrm{Der}^{\beta_k^{\ell}}(N_k,N_{k+\ell})$ for all $k=1,\ldots,n-\ell$ which further satisfy the relations $X(p_i\cdot p_j)=\beta_j^{i+\ell}\big(X^i(p_i),p_j\big)+(-1)^{\ell i }\beta_{j+\ell}^i\big(p_i,X^{j}(p_j)\big)$ for all $i+j\leq n-\ell.$

    \item Degree $n-1$: are pairs $X_{n-1}=\big(X_{n-1}^A,X^1\big)$ with $X_{n-1}^A\in D(A,N_{n-1})$ and $X^1\in\mathrm{Der}^{\beta_{1}^{N-1}}(N_1,N_n).$

    \item Degree $n$: are simply $N_n$-valued derivations of $A.$
\end{itemize}
\end{lemma}
\begin{proof}
Straightforward via the Leibniz rule for graded derivations.
\end{proof}
We should notice that for a degree $k$ derivation $X_k$, each component $X^{\ell}$ has generalized vector valued-symbol, which in the case that the family $\beta=\{\beta_j^i\}$ are all non-degenerate, is determined uniquely by the \emph{same} operator $X_k^A\in D(A,N_k).$
\vspace{1mm}

In the remainder of this paper, we will describe and comment on these classes of operators in the context of both diole and triole algebras and show how some usual notions arising in geometry appear from their descriptions.

\subsection{Aspects of Diolic Differential Calculus}
Given a diole algebra $\EuScript{A}$ and an $\EuScript{A}$-module $\EuScript{P}$, one may easily describe those degree zero elements of
$D_1(\EuScript{A},\EuScript{P})$ as tuples $X=(X^A,X^P)$ with $X^A\in D(A,P_0)$ and $X^P\in \mathrm{Diff}_1(P,P_1)$ such that 
$$X(a\cdot p)=\beta\big(X^A(a)\otimes p)+aX^P(p),$$
for all $a\in A,p\in P.$ 
These classes of operators appear in a variety of familiar situations.

\begin{example}
The pair $\delta:=(\delta_0^A,\delta_0^P)$ given by $\delta_0^A=d_{\mathrm{dR}}:A\rightarrow \Omega^1(A)$ and $\delta_0^P:=j_1^P:P\rightarrow \mathcal{J}^1(P)$ is an element of $\delta\in D_1\big(\EuScript{A},\Omega^1(\EuScript{A})\big)$.
\end{example}
\begin{example}
A covariant derivative $d_{\nabla}$ associated to a linear connection $\nabla$ in $P$ is an element of $D_1\big(\EuScript{A},\EuScript{P}_{\Omega^1}\big)$ i.e. $d_{\nabla}(ap)=da\cdot p+ad_{\nabla}(p),$ for $a\in A,p\in P.$
\end{example}
\begin{example}
\label{Example: Der operators along maps}
Consider the diole module from Example \ref{Example: Diole Along Map}. Then objects of $D_1(\EuScript{A},\EuScript{P}_{\overline{\varphi}})$ describe \emph{Der operators along $\overline{\varphi}$} i.e. operators $\Delta\in \mathrm{Diff}_1(P,Q)$ satisfying 
$\Delta(ap)=X(a)\cdot \overline{\varphi}(p) +\varphi(a)\cdot \Delta(p).$
\end{example}

\subsubsection{An Atiyah-like Sequence and Differential Operators}
A certain Atiyah-like sequence may be found by the description of the $A$-modules in Example \ref{Example: Der operators along maps}. 
\begin{theorem}
\label{Theorem: Atiyah-Sequence Along Maps}
Consider $D_1\big(\EuScript{A},\EuScript{P}_{\overline{\varphi}}\big)$ as above. If the $B$-submodule generated by $\mathrm{Im}(\overline{\varphi})\subset Q$ is faithful, then there is a surjective `symbol' map
$\sigma_{\overline{\varphi}}:D_1(\EuScript{A},\EuScript{P}_{\overline{\varphi}})_0\rightarrow D(A)_{\varphi},$
and an exact sequence of $A$-modules
$$0\rightarrow \mathrm{Hom}_A\big(P,Q^{<\varphi}\big)\hookrightarrow D_1(\EuScript{A},\EuScript{P}_{\overline{\varphi}})_0\xrightarrow{\sigma_{\overline{\varphi}}}D_1(A)_{\varphi}\rightarrow 0.$$
Moreover, if there exists a linear connection in $P$, then there exists a connection along $\overline{\varphi}$.
\end{theorem}
\begin{proof}
This follows from the general technical observation that a derivation in a pull-back bundle is equivalently describe in terms of a certain Der-operator along the bundle map. 
Then by forming the (artifical) diole algebra associated to a pull-back bundle as in \cite{Diole1}, the result follows from the general existence of the so-called diolic Atiyah sequence established in \emph{loc.cit.}

\end{proof}

Here we have denoted the module of derivations along $\varphi$ by $D_1(A)_{\varphi}:=\{X:A\rightarrow B|X(a\cdot_Aa')=X(a)\cdot_B\varphi(a')+\varphi(a)\cdot_B X(a')\}$ for $a,a'\in A.$ 

\begin{remark}
If $\nabla$ is the connection in $P$, the unique linear connection along $\overline{\varphi}$, denoted $\nabla_{\overline{\varphi}}$ is determined for all $X\in D(A)$ by 
$\nabla_{\overline{\varphi}}(X):=\overline{\varphi}\circ \nabla_X.$ Consequently, we have the \emph{covariant differential along} $\overline{\varphi}$, given
$d_{\nabla_{\overline{\varphi}}}:P\rightarrow \Omega^1(P)_{\overline{\varphi}},\hspace{2mm} d_{\nabla_{\overline{\varphi}}}=\overline{\varphi}\circ d_{\nabla}.$
The calculus of linear connections along maps deserves to be systematically developed and we see here the formalism of diolic commutative algebra provides a nice environment to do so. 
\end{remark}

Generalizations of these operators along maps exist in this context, and are naturally described within the diolic formalism in terms of the functor $\mathrm{Diff}_{k}(\EuScript{A},-)$ applied to the diole module $\EuScript{P}_{\overline{\varphi}}.$

\begin{lemma}
\label{Theorem: k-th order diolic differential operators along map}
A degree zero operator $\Delta\in \mathrm{Diff}_k(\EuScript{A}_{\pi},\EuScript{P}_{\overline{\varphi}})$ is described by a pair $\Delta=(\Delta^A,\Delta^P)$ such that $\Delta^A\in \mathrm{Diff}_k(A)_{\varphi}$ and $\Delta^P\in \mathrm{Diff}_k(P,Q^{<\varphi})$ such that 
$$\overline{\varphi}(p)\cdot \delta_a^k(\Delta^A)(1_A)=\delta_a^k(\Delta^P)(p).$$
\end{lemma}

Here, $\mathrm{Diff}_k(A)_{\varphi}$ are $k$-th order differential operators along $\varphi:A\rightarrow B.$ For instance,
$\Delta \in \mathrm{Diff}_1(A)_{\varphi}$ is a $\mathbf{k}$-linear map $\Delta:A\rightarrow B$ such that 
$$\Delta(a_0\cdot_A a_1)=\varphi(a_0)\cdot_B\Delta(a_1)+\varphi(a_1)\cdot_B\Delta(a_0)+\varphi(a_0\cdot_A a_1)\cdot_B\Delta(1_A),\hspace{1mm} a_0,a_1\in A.$$
Notice that $D_1(A)_{\varphi}=\{\Delta\in \mathrm{Diff}_1(A)_{\varphi}|\Delta(1_A)=0\}.$

Differential operators over smooth maps play an important role in many areas. For example, some of their interesting variants such as formal, pseudo and $\hbar$-versions are related to “quantization of symplectic micromorphisms” \cite{Thick}. Surprisingly, they arise not from external modifications to the functor $\mathrm{Diff}_k$, but by merely finding the `right' environment in which to compute them.

\subsection{Aspects of Triolic Differential Calculus}
Derivations in the graded commutative algebra $\EuScript{T}$ have an important interpretation in the context of symmetries of the underlying bilinear form $g.$

\begin{lemma}
\label{lemma: Degree Zero Triolic Derivation}
A triolic derivation $X_0\in D_1^{\mathrm{tri}}(\EuScript{T})$ is described by a triple $X_0=\big(X_0^A,X_0^P,X_0^Q\big)$ of operators with $X_0^A\in D(A)$ an ordinary derivation, and with $X_0^P\in \mathrm{Der}(P),X_0^Q\in \mathrm{Der}(Q)$ two Der-operators with shared scalar symbol $\sigma(X_0^P)=\sigma(X_0^Q)=X_0^A,$ such that $X_0^Q\big(g(p_1,p_2)\big)=g\big(X_0^P(p_1),p_2\big)+g\big(p_1,X_0^P(p_2)\big).$
\end{lemma}

Lemma \ref{lemma: Degree Zero Triolic Derivation} shows that 
the datum of a degree zero derivation $X_0=(X_0^A,X_0^P,X_0^Q)$ is equivalent to a $g$-preserving pair of connections $(\nabla,\Delta)$, as in subsection \ref{sssec: Vector-Valued Forms}.

\begin{remark}
In other words, a pair $(\EuScript{T},X_0)$ is equivalent to supplying an $A$-module $P$ equipped with an inner (gauge) structure $-$ a gauge vector-valued bilinear form.
\end{remark}

By similar considerations, we have the following.

\begin{lemma}
\label{Degree1TriolicDerivations}
A degree $1$ triolic derivation is a pair $X_1=\big(X_1^A,X_1^P\big)$ of operators with $X_1^A \in D(A,P)$ and with $X_1^P:P\rightarrow Q,$ an element of $\mathrm{Diff}_1(P,Q)$ which satisfies $X_1^P(1)=0,$ and  $X_1^{P}(ap)=g\big(X_1^A(a),p\big)+aX_1^P(p),$ for all $a\in A,p\in P.$  
\end{lemma}

The class of operators described in Lemma \ref{Degree1TriolicDerivations} appear to be new, and provide a generalization of the notion of a derivation in a vector bundle in the case where we have a vector-valued fiber metric. In more familiar terms, one may extract their definition as follows.

Let $E\rightarrow M$ and $F\rightarrow M$ be two vector bundles of ranks $m_P,m_Q$ with modules of sections $P,Q,$ respectively.
Suppose that $\gamma:\Gamma(E)\otimes \Gamma(E)\rightarrow \Gamma(F)$ is a non-degenerate vector valued fiber metric.
An $\mathbb{R}$-linear operator $\Box:P\rightarrow Q$ is said to be a \emph{Der-operator for} $\gamma:P\otimes P\rightarrow Q$, if
\begin{enumerate}
    \item It is an additive operator,
    
    \item It satisfies the Der-Lebiniz-like rule 
    \begin{equation}
        \label{VectorvaluedGammaDerLeibniz}
    \Box(ap)=\gamma\big(X(a)\otimes p\big)+a\Box(p),
    \end{equation}
for $X$, some $P$-valued derivation of $A,$ called the \emph{vector-valued symbol}, with $p\in P,a\in A.$
\end{enumerate}
Relation (\ref{VectorvaluedGammaDerLeibniz}) might be called the $Q$-\emph{valued} $\gamma$-\emph{Der-Leibniz rule}, or simply the $\gamma$-\emph{Der-Leibniz} rule when $Q$ is understood.
The analogy with the ordinary Der-Leibniz rule is clear by setting $X(a)\cdot_{\gamma}p:=\gamma\big(X(a)\otimes p\big)\in Q.$ 

It remains to be seen if such operators appear elsewhere. 
\vspace{3mm}

\subsubsection{An Atiyah-like Sequence and Differential Operators}
Consider a triolic derivation $X_0$, as described by Lemma \ref{lemma: Degree Zero Triolic Derivation} and the natural projection $\sigma:D_1(\EuScript{T})_0\rightarrow D(A),$ defined by $\sigma(X_0):=X_0^A.$ 
\vspace{.5mm}

It is straightforward to observe that the $A$-module $\mathrm{ker}(\sigma)$ is spanned by pairs $(X,Y)$ where $X:P\rightarrow P$ is an endomorphism of $P$, and $Y:Q\rightarrow Q$ is an endomorphism of $Q,$ satisfying $Y\big(g(p_1,p_2)\big)=g\big(X(p_1),p_2\big)+g\big(p_1,X(p_2)\big).$ Denote their totality by
$$\mathcal{E}\mathrm{nd}\big(g;P,Q\big):=\{\text{pairs}, X_0^P\in\mathrm{End}_A(P),X_0^Q\in \mathrm{End}_A(Q)| X_0^Q\circ g=g\circ (X_0^P\otimes \mathbf{1}_P+\mathbf{1}_P\otimes X_0^P)\}.$$

In the case when $P$ and $Q$ are projective and finitely generated we find a short-exact sequence that we call the triolic Atiyah sequence (c.f. sequence \ref{eqn: Atiyah sequence}),
\begin{equation}
\label{eqn: Triole Atiyah sequence}
    0\rightarrow \mathcal{E}\mathrm{nd}\big(g;P,Q\big)\rightarrow D_1(\EuScript{T})_0\xrightarrow{\sigma} D(A)\rightarrow 0.
    \end{equation}
 
As a final result, let us describe $k$-th order differential operators of degree zero over a triole algebra.

\begin{theorem}
\label{Theorem: Triolic Differential Operators}
Degree zero differential operators over $\EuScript{T}$ i.e. elements of
$\mathrm{Diff}_k(\EuScript{T})_0$ are triples $\Delta_0=(\Delta_0^A,\Delta_0^P,\Delta_0^Q)$ where $\Delta_0^A\in \mathrm{Diff}_k(A),\Delta_0^P\in \mathrm{Diff}_k(P,P),\Delta_0^Q\in \mathrm{Diff}_k(Q,Q)$ which satisfy the following relations for $a\in A$ $p_0,p_1,p\in P$ and $q\in Q,$

\begin{itemize}
\item $p\delta_{a}^k(\Delta_0^A)(1_A)=\delta_a^k(\Delta_0^P)p,$

\item $
q\delta_a^k(\Delta_0^A)(1_A)=\delta_a^k(\Delta_0^Q)q,$

\item $g\big(p_0,\delta_{a}^{k}(\Delta_0^P)(p_1)\big)=\delta_a^k(\Delta_0^Q)g(p_0,p_1).$
\end{itemize}
Moreover $\Delta_0$ satisfies the following compatibility relation for all $a\in A,p_0,p_1\in P$,
\begin{equation}
    \label{DiolicDiffCompatibilityRelation}
\delta_a^{k-1}\Delta_0^Qg(p_1,p_0)=g\big(p_1,\delta_a^{k-1}\Delta_0^P(p_0)\big)+g\big(p_0,\delta_a^{k-1}\Delta_0^P(p_1)\big)+g(p_0,p_1)\big[\delta_a^{k-1}\Delta_0^A\big](1_A),\end{equation}
where we set $\delta_a^{k-1}:=\delta_a\circ \ldots\circ \delta_a,$ where we have $(k-1)$ factors of $\delta_a:=[-,a].$
\end{theorem}
Notice that the $k$-th order differential operators $\Delta_0^P,\Delta_0^Q$ have the shared scalar symbol, determined by the scalar operator $\Delta_0^A$. Thus, proceeding as above, one may argue the existence of a projection map 
$$\xi:\mathrm{Diff}_{k}(\EuScript{T})_0\rightarrow \mathrm{Diff}_k(A), \hspace{5mm} \xi(\nabla_0^A,\nabla_0^P,\nabla_0^Q):=\nabla_0^A.$$
Characterizing its kernel indicates that
such an $A$-module $\ker(\xi)$ consist of pairs of differential operators $(\nabla_0^P,\nabla_0^Q)$ with $\nabla_0^P\in\mathrm{Diff}_{k-1}(P,P)$ and $\nabla_0^Q\in \mathrm{Diff}_{k-1}(Q,Q)$ which satisfy the relation \begin{equation}
    \label{Diff(gPQ)relation}
\delta_a^{k-1}\nabla_0^Qg(p_1,p_2)=g\big(\delta_a^{k-1}\nabla_0^P(p_1),p_2\big)+g\big(p_1,\delta_a^{k-1}\nabla_0^P(p_2)\big).
\end{equation}

Denoting the $A$-module $\ker(\xi)$ by $\mathcal{D}\mathrm{iff}_{k-1}\big(g;P,Q\big)$, one may interpret elements of this space of differential operators as a type of generalized symmetry of the underlying gauge structure $g,$ as explained for instance in \cite{Triole1}. 
A corresponding $g$-invariant Hamiltonian formalism may then be deduced by studying the associated `symbol' modules $\mathcal{D}\mathrm{iff}_{k-1}(g;P,Q)/\mathcal{D}\mathrm{iff}_{k-2}(g;P,Q).$

\section{Conclusion}
In this work we traversed through a unique perspective on the modern mathematics used in modelling physics, largely due to the insights of Alexandre Vinogradov. 
By understanding geometric structures in purely algebraic terms, and by describing the functors of differential calculus over interesting species of graded algebras, we have not only found a common origin for many interesting structures and classes of operators arising in algebra and geometry, but have also discovered several new features. 
\subsection{Future Directions}
\label{ssec: Future Directions} 
The idea in gauge theory that two fields are physically indistinguishable if one of them is obtained from another by a gauge transformation appears within the algebraic approach to the calculus of linear connections sketched in this work as follows.

Let $(P,\nabla)$ be an $A$-module with a Der-operator and suppose $\widetilde{\nabla}$ is another Der-operator in $P$ such that $\sigma(\nabla)=\sigma(\widetilde{\nabla})=X\in D(A).$ One says that $\nabla$ and $\widetilde{\nabla}$ are gauge equivalent if there exists $\varphi\in \mathrm{Aut}(P)$ such that 
$$\widetilde{\nabla}=\varphi\circ \nabla\circ \varphi^{-1}.$$
In this case the pairs $(P,\nabla)$ and $(P,\widetilde{\nabla})$ are also said to be equivalent.

The corresponding notion of gauge equivalence for linear connections is inherited from this definition; linear connections $\nabla,\widetilde{\nabla}$ are gauge equivalent if there exists an automorphism $\varphi\in \mathrm{Aut}(P)$ such that $\widetilde{\nabla}_X=\varphi\circ \nabla_X\circ \varphi^{-1}$ for all $X\in D(A).$

\begin{remark}
A more general notion of gauge equivalence is possible. Namely, let $P,Q\in \mathbf{Mod}(A)$ with linear connections $\nabla,\Box$, respectively. There is a gauge equivalence $(P,\nabla)\sim (Q,\Box)$ if there exists an isomorphism $\overline{\varphi}:P\rightarrow Q$ of $A$-modules such that $\Box_X=\overline{\varphi}\circ \nabla_X\circ \overline{\varphi}^{-1}.$ 
\end{remark}
A \emph{gauge $A$-module} is a triple $(P,\nabla,\Xi)$ consisting of a projective $A$-module $P$, a flat linear connection $\nabla$ and a gauge structure $\Xi\in \mathcal{T}(P).$ It is of \emph{type} $(p,q)$ when $\Xi\in P_q^p\subseteq \mathcal{T}(P).$ 

A suitable notion of gauge equivalence for these triples can then be introduced along the following lines.
Call a $(p,q)$-gauge module \emph{orientable} if the top dimensional wedge power is $1$-dimensional and consider $\mathfrak{o}=\mathfrak{o}(P,\Xi),\mathrm{SO}(P,\Xi)$ and the infinitesimal counterpart $\mathfrak{so}(P;\Xi).$

We then expect the following definition to play an important role in the classification of gauge $A$-modules of a given type via cohomological techniques.
\begin{definition}
Let $\widetilde{\nabla}$ and $\nabla$ be linear connections in $P$ which preserve the inner structure $\Xi$ of type $(p,q).$ They are said to be \emph{$\Xi$-gauge equivalent} if there exists some $\overline{\varphi}\in \mathrm{SO}(P;\Xi)$ such that 
$\widetilde{\nabla}_X=\overline{\varphi}\circ \nabla_X\circ\overline{\varphi}^{-1},$ for all $X\in D(A).$
\end{definition}

We write $\widetilde{\nabla}\sim_{\Xi}\nabla$ to indicate gauge equivalence with respect to the structure $\Xi$ and notice the following.
\begin{lemma}
The space $\mathrm{Conn}_{\Xi}(P)$ of $\Xi$-preserving linear connections in $P$ is an affine space and $\sim_{\Xi}$ yields a well-defined equivalence relation.
\end{lemma}

\subsubsection{Some final comments}
We believe these notions provide first steps in obtaining a general classification of gauge modules of a given type. It is likely to be achieved by means of the theory of \emph{special characteristic classes} \cite[{\normalfont\S 4.5.31}]{VinFat}.

Indeed, in \emph{loc.cit.} this was achieved for $1$-dimensional gauge modules of type $(0,2)$. In other words, the situation treated aims to classify line bundles supplied with a fiber-wise bilinear form.

We believe it would be interesting to carry out the analysis for finite dimensional gauge modules of arbitrary type $(p,q)$.
\vspace{1mm}

Furthermore, it would be interesting to develop the theory of gauge manifolds of general type in the context of more general graded algebraic structures, for instance $n$-ole algebras, in the same way a gauge manifold of type $(0,2)\times (1,0)$ appears in terms of triolic differential calculus.

\subsection{Ackowledgements}
I would like to thank Volodya Rubtsov for supervising my doctoral thesis as a member of MathSTIC and LAREMA at the University of Angers, France, where partial results in this work were obtained \cite{Kry}.

It is my pleasure to also thank the organizers of the conference: \emph{Diffieties, Cohomological Physics, and Other Animals} for granting me the opportunity to give a talk.
It is a great honour to be part of and interact with the impressive scientific community related to the remarkable work of Professor Vinogradov.

\bibliographystyle{amsalpha}
\bibliography{Bibliography}
\end{document}